\documentclass[11pt,dvips,twoside]{article}

\usepackage{pslatex}
\usepackage{fancyhdr}
\usepackage{graphicx}
\usepackage{geometry}

\RequirePackage{amsfonts,amssymb,amsmath,amscd,amsthm}
\RequirePackage{txfonts}
\RequirePackage{graphicx}
\RequirePackage{xcolor}
\RequirePackage{geometry}
\RequirePackage{enumerate}

\def\figurename{Figure} 
\makeatletter
\renewcommand{\fnum@figure}[1]{\figurename~\thefigure.}
\makeatother

\def\tablename{Table} 
\makeatletter
\renewcommand{\fnum@table}[1]{\tablename~\thetable.}
\makeatother

\usepackage{amsmath}
\usepackage{amssymb}
\usepackage{amsfonts}
\usepackage{amsthm,amscd}

\newtheorem{theorem}{Theorem}[section]
\newtheorem{lemma}[theorem]{Lemma}
\newtheorem{corollary}[theorem]{Corollary}
\newtheorem{proposition}[theorem]{Proposition}
\theoremstyle{definition}
\newtheorem{definition}[theorem]{Definition}
\newtheorem{example}[theorem]{Example}

\theoremstyle{remark}
\newtheorem{remark}[theorem]{Remark}

\numberwithin{equation}{section}

\begin{document}

\title{\bfseries\scshape{Hom-Novikov color algebras}}

\author{\bfseries\scshape Ibrahima BAKAYOKO\thanks{e-mail address: ibrahimabakayoko27@gmail.com}\\
D\'epartement de Math\'ematiques,\\
Centre Universitaire de N'Z\'er\'ekor\'e/UJNK\\
BP 50 N'Z\'er\'ekor\'e, Guin\'ee.
}
 
\date{}
\maketitle 


\noindent\hrulefill

\noindent {\bf Abstract.} 
The aim of this paper is to introduce Hom-Novikov color algebra and give some constructions of Hom-Novikov color algebras from a given one and a 
(weak) morphism. Other interesting
constructions using averaging operators, centroids, derivations and tensor product are given. We also
proved that any Hom-Novikov color algebra is Hom-Lie admissible.
Moreover, we introduce Hom-quadratic Hom-Novikov color algebras and provide some properties by twisting.
 It is also shown that the Hom-Lie color algebra associated to a given quadratic Hom-Novikov color algebra is also quadratic.

\noindent \hrulefill

\vspace{.3in}


\vspace{.08in} \noindent \textbf{Keywords}: Hom-associative color algebra, averaging operator, centroids,  Hom-Novikov color algebra, twisting,
Hom-quadratic Hom-Novikov colon algebras.
\vspace{.3in}
\vspace{.2in}
\section{Introduction}
Novikov algebras are algebraic structures such that the left operator multiplication form a Lie algebra and the right operator multiplication
are commutative. They were introduced by Balinskii and Novikov in the studies of Hamiltonian operators and Poisson bracket with hydrodynamic type
\cite{BN1, BN2}. The theoretical study of Novikov algebras was started by Zel'manov \cite{Z} and Filipov \cite{F}. And the term ``Novikov algebras''
was first used by Osborn \cite{O}. Novikov algebras are connected to many areas of mathematical physics and geometry 
including Lie groups \cite{BM1}, Lie algebras \cite{BM2, BM3, BDV}, affine manifolds \cite{K}, convex homogeneous cones \cite{V},
 rooted tree algebras \cite{C}, vector fields \cite{B}, and vertex and conformal algebras \cite{BR, KV}.
In particular, they are Lie-admissible algebras, which are important in some physical applications, such as quantum mechanics and hadronic
 structures \cite{MH1, MH2, MH3, SR}.
Tortken algebras \cite{DA} are defined by a polynomial identity of degre four. Among other results, the author of \cite{DA} prove that the
 Jordan product on any Novikov algebra satisfies the Tortken identity. This result has been extended to the color case by Gao X. and Xu L. in \cite{GX}.

Hom-Lie algebras originate from the work of Hartwig J., Larsson D. and Silvestrov S. in \cite{HLS} but Hom-associative algebras was first 
introduced by Makhlouf A. and Silvestrov S. in \cite{MS} as the twisted version of Hom-Lie algebras and Hom-associative algebras repectively. 

A twisted generalization of Novikov algebras were introduced in \cite{DYN}, \cite{DYN2} and studied by Yuan L. and You H. in \cite{YY} and called Hom-Novikov algebras.
 Yau D. \cite{DYN} shown that Hom-Novikov algebras can be obtained from Novikov algebras by twisting along any algebra endomorphisms.
He computed all algebra endomorphisms on complex Novikov algebras of dimensions two and three, and described explicitly
 their associated Hom-Novikov algebras.
A construction due to Dorfman and Gel'fand is used to give constructions of Hom-Novikov from Hom-commutative algebras together with a derivation.
 Two other classes of Hom-Novikov algebras are
constructed from Hom-Lie algebras together with a suitable linear endomorphism, generalizing a
construction due to Bai and Meng

Hom-type algebraic structures of many other classical structures were studied such as G-Hom-associative algebras \cite{AS3}, Hom-alternative
 algebras, Hom-Malcev algebras and Hom-Jordan algebras \cite{DY3}, and
 Constructions of n-Lie algebra and Hom-Nambu-Lie algebras \cite{AMS}.

Classical (or ordinary) algebras were extended to generalized (or color or graded) algebras. Many works has been done in this direction,
among which one can cite : Simple Jordan color algebras arising from associative graded algebras \cite{BG}, 
 Representations and cocycle twists of color Lie algebras \cite{CSO},
 Derivations and extensions of Lie color algebras \cite{QY}, On the classification of 3-dimensional coloured Lie algebras \cite{SD}.
These structures are well-known to physicists and to mathematicans studying differential geometry and homotopy theory.
They were extended to the Hom-setting by studying Hom-Lie superalgebras, Hom-Lie admissible superalgebras
 in  \cite{FAAM} and  Hom-Lie color algebras \cite{LY}. 
 Modules over color Hom-Poisson algebras were introduced in \cite{BI1}. 
 Some color Hom-algebras, such as Hom-associative color algebras, Hom-Lie color algebras, Hom-dendriform color algebras, 
Hom-left-symmetric color algebras, Hom-Leibniz-Poisson color algebras and Hom-Poisson color algebras were studied in \cite{BD} under the name 
of generalized Hom-algebras. Moreover, some other color algebra structures, such as  tridendriform color algebras, pre-Poisson color algebras,
 post-Poisson color algebras, Hom-associative color (di)algebras, Hom-left-symmetric color dialgebras, were studied in \cite{B2}.

The purpose of this paper is to give various constructions of Hom-Novikov color algebras and  Hom-quadratic Hom-Novikov color algebras.
The paper is organized as follows. In section two, we recall basic notions concerning Hom-associative color algebras, Hom-Lie color algebras,
 averaging operators and centroids.
In section three, we deal with Hom-Novikov color algebras. We provide some constructions of Hom-Novikov color algebras by twisting.
We point out that the direct sum  of two Hom-Novikov algebras is also a Hom-Novikov algebra and the tensor product of a commutative
Hom-associative color algebra and a Hom-Novikov color algebra is a Hom-Novikov color algebra. 
We also prove that any Hom-Novikov color algebra is Hom-Lie admissible.
In section four, we introduce Hom-quadratic Hom-Novikov color algebras. We prove that Hom-quadratic Hom-Novikov color algebras are closed under 
twisting.  We also show that  any quadratic Hom-Novikov color algebra carries a structure of quadratic Hom-Lie color algebra.

Let us fix some notations and conventions :

i) $\sum_{x, y, z}$ means cyclic summation.

ii) Throughout this paper, all graded vector spaces are assumed to be over a field $\mathbb{K}$ of characteristic different from 2.

\pagestyle{fancy} \fancyhead{} \fancyhead[EC]{Ibrahima Bakayoko} 
\fancyhead[EL,OR]{\thepage} \fancyhead[OC]{Hom-Novikov color algebras} \fancyfoot{}
\renewcommand\headrulewidth{0.5pt}

\section{Preliminaries}
In this section, we give the definitions of Hom-associative color algebras, Hom-Lie color algebras,
 averaging operators, centroids and Rota-Baxter Hom-Lie color algebra. 

\begin{definition}
 \begin{enumerate}
  \item [1)] Let $G$ be an abelian group. A vector space $V$ is said to be a $G$-graded if, there exists a family $(V_a)_{a\in G}$ of vector 
subspaces of $V$ such that
$$V=\bigoplus_{a\in G} V_a.$$
\item [2)] An element $x\in V$ is said to be homogeneous of degree $a\in G$ if $x\in V_a$. We denote $\mathcal{H}(V)$ the set of all homogeneous elements
in $V$.
\item [3)] Let $V=\oplus_{a\in G} V_a$ and $V'=\oplus_{a\in G} V'_a$ be two $G$-graded vector spaces. A linear mapping $f : V\rightarrow V'$ is said 
to be homogeneous of degree $b$ if 
$$f(V_a)\subseteq  V'_{a+b}, \forall a\in G.$$
If, $f$ is homogeneous of degree zero i.e. $f(V_a)\subseteq V'_{a}$ holds for any $a\in G$, then $f$ is said to be even.
 \end{enumerate}
\end{definition}
\begin{remark}
Let $V=\oplus_{a\in G} V_a$ and $V'=\oplus_{a\in G} V'_a$ be two $G$-graded vector spaces. The tensor product $V\otimes V'$ is also a $G$-graded vector space such that for $\alpha\in G$, we 
$$(V\otimes V')_\alpha=\sum_{\alpha=a+a'}V_a\otimes V_{a'}.$$
\end{remark}
\begin{definition}
 Let $G$ be an abelian group. A map $\varepsilon :G\times G\rightarrow {\bf \mathbb{K}^*}$ is called a skew-symmetric bicharacter on $G$ if the following
identities hold, 
\begin{enumerate} 
 \item [(i)] $\varepsilon(a, b)\varepsilon(b, a)=1$,
\item [(ii)] $\varepsilon(a, b+c)=\varepsilon(a, b)\varepsilon(a, c)$,
\item [(iii)]$\varepsilon(a+b, c)=\varepsilon(a, c)\varepsilon(b, c)$,
\end{enumerate}
$a, b, c\in G$,
\end{definition}
\begin{remark}
Observe that $\varepsilon(a, 0)=\varepsilon(0, a)=1, \varepsilon(a,a)=\pm 1 \;\mbox{for all}\; a\in G, \;\mbox{where}\; 0 \;\mbox{is the identity of}\; G.$
\end{remark}
If x and y are two homogeneous elements of degree $a$ and $b$ respectively and $\varepsilon$ is a skew-symmetric bicharacter, 
then we shorten the notation by writing $\varepsilon(x, y)$ instead of $\varepsilon(a, b)$.

\begin{definition}
 By a color Hom-algebra we mean a quadruple $(A, \cdot, \varepsilon, \alpha)$ in which $A$ is a $G$-graded vector space i.e. $A=\bigoplus_{a\in G}A_a$, 
$\cdot : A \times A \rightarrow A$ is an even bilinear map i.e. $A_a\cdot A_b\subseteq A_{a+b}$, for all $a, b\in G$, $\varepsilon : G\times G\rightarrow\mathbb{K}^*$ is a bicharacter and 
$\alpha : A\rightarrow A$ is an even linear map.

If in addition $x\cdot y=\varepsilon(x, y)y\cdot x$, for any $x, y\in\mathcal{H}(A)$, the  color Hom-algebra 
$(A, \cdot, \varepsilon, \alpha)$ is said to be $\varepsilon$-commutative.
\end{definition}

\begin{definition}
 A derivation of degree $d\in G$ on a color Hom-algebra $(A, \cdot, \varepsilon, \alpha)$ is a linear map $D : A\rightarrow A$ such that for 
any homogeneous element $x$ and any $y\in A$ one has
$$D(x\cdot y)=D(x)\cdot y+\varepsilon(d, x) x\cdot D(y).$$
In particular, an even derivation $D : A\rightarrow A$ is a derivation of degree zero i.e. for any homogeneous elements $x, y\in A$,
$$D(x\cdot y)=D(x)\cdot y+x\cdot D(y).$$
\end{definition}

\begin{definition}
 Let $(A, \cdot, \varepsilon, \alpha)$ and $(A', \cdot', \varepsilon, \alpha')$ be two color Hom-algebras.
An even linear map $f : (A, \cdot, \varepsilon, \alpha)\rightarrow (A', \cdot', \varepsilon, \alpha')$ is called  a weak morphism  if,
for any $x, y\in A$, $$f(x\cdot y)=f(x)\cdot'f(y).$$
If in addition $f\circ \alpha=\alpha'\circ f$, $f$ is said to be a morphism of color Hom-algebras.
\end{definition}
\begin{definition}
A color Hom-algebra $(A, \cdot, \varepsilon, \alpha)$ is said to be   
\begin{enumerate}
 \item [a)]
{\it multiplicative} if $\alpha$ is a morphism for $\cdot$,
\item [b)]
{\it regular} if $\alpha$ is an automorphism for $\cdot$,
\item [c)]
{\it involutive} if $\alpha^2=id_A$.
\end{enumerate}
\end{definition}
Motivated by the notions of, left and right unit,  left and right commutativity, left and right symmetry, left and right nondegenerate,and so on, we have the 
following Definition. 
\begin{definition}\cite{B2}\label{bk3}
A {\it left  averaging operator} (resp. {\it right averaging operator}) over a color Hom-algebra  $(A, \cdot, \varepsilon, \alpha)$ is an even 
linear map $\beta : A\rightarrow A$ such that $\alpha\circ\beta=\beta\circ\alpha$ and
$$\beta(x)\cdot \beta(y)=\beta(\beta(x)\cdot y),\quad \Big(\mbox{resp.}\quad \beta(x)\cdot \beta(y)=\beta(x\cdot \beta(y))\Big),\;\;
\mbox{for all}\; x, y\in\mathcal{H}(A).$$
An {\it averaging operator} over a color Hom-algebra $A$ is both {\it left} and {\it right averaging operator} i.e.
$$\beta(\beta(x)\cdot y)= \beta(x)\cdot \beta(y)=\beta(x\cdot \beta(y)).$$
\end{definition}
Similarly, 
\begin{definition}\cite{B2}\label{ctr}
Let $(A, \cdot, \varepsilon, \alpha)$ be a color Hom-algebra. 
 An even linear map $\beta : A\rightarrow A$  is said to be
 an element of the {\it left centroid} (resp. {\it right centroid}) if $\alpha\circ\beta=\beta\circ\alpha$ and
$$\beta(x\cdot y)=\beta(x)\cdot y \quad \Big(\mbox{resp.}\quad \beta(x\cdot y)=x\cdot \beta(y)\Big),\;\;\mbox{for any}\; x, y\in\mathcal{H}(A).$$
An even linear map $\beta : A\rightarrow A$ on $A$ which is both an element of the {\it left centroid} and an element of the {\it right centroid}
is called an element of the {\it centroid} i.e.
$$\beta(x)\cdot y=\beta(x\cdot y)=x\cdot \beta(y).$$
\end{definition}
\begin{remark}
For $\varepsilon$-commutative and $\varepsilon$-skew-symmetric color Hom-algebras the three notions :
{\it left  averaging operator}, {\it right averaging operator} and {\it averaging operator} (resp. {\it left centroid},
 {\it right centroid} and {\it centroid}) coincide.
\end{remark}

\begin{definition}\cite{LY}
A Hom-associative color algebra is a color Hom-algebra $(A, \cdot, \varepsilon, \alpha)$  such that
\begin{eqnarray}
 \alpha(x)\cdot(y\cdot z)=(x\cdot y)\cdot \alpha(z), \label{hass}
\end{eqnarray}
for any $x, y, z\in\mathcal{H}(A)$.
\end{definition}

\begin{remark}
When $\alpha=Id$, we recover the classical associative color algebra.
\end{remark}
This proposition is the color version of (\cite{DYN} Lemma 4.1).
\begin{proposition}
 Let $(A, \cdot, \varepsilon, \alpha)$ be a commutative Hom-associative color algebra and 
$\xi\in A_0$ (subset of elements of degre $0$). Let us define new operation $\ast : A\times A\rightarrow A$  on $A$ by 
$$x\ast y:=\xi\cdot(x\cdot y)$$
for any $x, y\in\mathcal{H}(A)$. Then $A'=(A, \ast, \varepsilon, \alpha^2)$ is a Hom-associative color algebra. Moreover, if $A$ is multiplicative,
$A'$ is also multiplicative.
\end{proposition}
\begin{proof}
 The multiplicativity and the $\varepsilon$-commutativity are trivial. Now prove the Hom-associativity of $A'$. For $x, y, z\in\mathcal{H}(A)$,
\begin{eqnarray}
  (x\ast y)\ast\alpha^2(z)
&=&\xi\cdot[(\xi\cdot(x\cdot y))\cdot\alpha^2(z)]\nonumber\\
&=&\xi\cdot[\alpha(\xi)\cdot((x\cdot y)\cdot\alpha(z))]\;\;(\mbox{by}\;\; (\ref{hass}))\nonumber\\
&=&\xi\cdot[\alpha(\xi)\cdot(\alpha(x)\cdot (y\cdot z))]\;\;(\mbox{by}\;\; (\ref{hass}))\nonumber\\
&=&\xi\cdot[(\xi\cdot\alpha(x))\cdot\alpha(y\cdot z)]\;\;(\mbox{by}\;\; (\ref{hass}))\nonumber\\
&=&\xi\cdot[(\alpha(x)\cdot\xi)\cdot\alpha(y\cdot z)]\nonumber\\
&=&\xi\cdot[\alpha^2(x)\cdot(\xi\cdot (y\cdot z))]\;\;(\mbox{by}\;\; (\ref{hass}))\nonumber\\
&=&\alpha^2(x)\ast(y\ast z)\nonumber.
\end{eqnarray}
This completes the proof.
\end{proof}

\begin{definition}\cite{LY}
 A Hom-Lie color  algebra is a color Hom-algebra $(A, [\cdot, \cdot], \varepsilon, \alpha)$  such that, for any $x, y, z\in\mathcal{H}(A)$,
\begin{eqnarray}
 [x, y]&=&-\varepsilon(x, y)[y, x]\qquad (\varepsilon\mbox{-skew-symmetry}),\label{ss}\qquad\\
\sum_{x, y, z}\varepsilon(z, x)[\alpha(x), [y, z]]&=&0\;\qquad\qquad\qquad
 (\varepsilon\mbox{-Hom-Jacobi identity}).\quad \label{chli}
\end{eqnarray}
\end{definition}

\begin{example}
It is clear that Lie color algebras are examples of Hom-Lie color algebras
by setting $\alpha = id$ . If, in addition, $\varepsilon(x, y)=1$ or $\varepsilon(x, y)=(-1)^{|x||y|}$, then the Hom-Lie color
algebra is nothing but a classical Lie algebra or Lie superalgebra. The Hom-Lie algebra and
Hom-Lie superalgebra are also obtained when $\varepsilon(x, y)=1$ and $\varepsilon(x, y)=(-1)^{|x||y|}$ respectively.
 See \cite{LY} for other examples.
\end{example}
\begin{definition}\cite{B2}
 A color Hom-algebra  $(A, \mu, \varepsilon, \alpha)$ is said to be a Hom-Lie  admissible color algebra if, for any hogeneous elements
$x, y\in A$, the bracket $[\cdot, \cdot] : A\times A\rightarrow A$ defined
by $$[x, y]=\mu(x, y)-\varepsilon(x, y)\mu(y, x)$$
satisfies the $\varepsilon$-Hom-Jacobi identity.
\end{definition}
\begin{example}
 Any Hom-associative color algebra and any Hom-Lie color  algebra are Hom-Lie admissible color algebras. 
\end{example}
\begin{definition}\cite{B2}
A  Rota-Baxter Hom-Lie color algebra of weight $\lambda\in\mathbb{K}$ is a Hom-Lie color 
algebra $(L, [-, -], \varepsilon, \alpha)$ equiped with an even linear map $R : L\rightarrow L$ that satisfies the identities
\begin{eqnarray}
R\circ\alpha&=&\alpha\circ R,\\
{[R(x), R(y)]} &=& R\Big([R(x), y] + [x, R(y)] +\lambda [x, y]\Big),\label{rbhl}
\end{eqnarray}
for all $x, y\in\mathcal{H}(L)$.
\end{definition}
\begin{definition}\cite{BD}
A color Hom-algebra $(S, \cdot, \varepsilon, \alpha)$
is called a Hom-left-symmetric color algebra if the following Hom-left-symmetric color identity 
(or {\it $\varepsilon$-Hom-left-symmetric identity}) 
\begin{eqnarray}
(x\cdot y)\cdot\alpha(z)-\alpha(x)\cdot(y\cdot z)=\varepsilon(x, y)\Big((y\cdot x)\cdot\alpha(z)-\alpha(y)\cdot(x\cdot z)\Big)\label{clsa}
\end{eqnarray}
is satisfied for all $x, y, z\in \mathcal{H}(S)$.
\end{definition}
\begin{lemma}\cite{B2}\label{hde}
 Let $(L, [-, -], \varepsilon, \alpha, R)$ be a Rota-Baxter  Hom-Lie color algebra of weight $0$. Then $L$ is a Hom-left-symmetric color algebra
 with $$x\ast y=[R(x), y],$$
for $x, y\in\mathcal{H}(L)$.
\end{lemma}
\section{Hom-Novikov color algebras}
In this section, we introduce Hom-Novikov color algebra and give some properties.
In particular we show that the $\varepsilon$-commutator of any Hom-Novikov color algebra give rises to Hom-Lie color algebra.
\begin{definition}
 A Hom-Novikov color algebra is a quadruple $(A, \cdot, \varepsilon, \alpha)$ consisting of a $G$-graded vector space $A$, an even bilinear map 
$\cdot : A \times A \rightarrow A$, a bicharacter $\varepsilon : G\times G\rightarrow\mathbb{K}^*$ and an even linear map $\alpha : A\rightarrow A$
satisfying 
 \begin{eqnarray}
  (x\cdot y)\cdot\alpha(z)&=&\varepsilon(y, z)(x\cdot z)\cdot \alpha(y)\label{n1}\\
(x\cdot y)\cdot\alpha(z)-\alpha(x)\cdot(y\cdot z)&=&\varepsilon(x, y)\Big((y\cdot x)\cdot\alpha(z)-\alpha(y)\cdot(x\cdot z)\Big),\label{n2}
 \end{eqnarray}
for any $x, y, z\in\mathcal{H}(A)$.
\end{definition}
\begin{remark}
 \begin{enumerate}
  \item [a)]
When $\alpha =id_A$, we recover Novikov color algebras \cite{GX}.
\item[b)]
 When $A$ is trivialy graded, we  recover Hom-Novikov algebras \cite{DYN}.
  \item[c)]
 When $\alpha =id_A$ and $A$ is trivialy graded, we recover Novikov algebras \cite{Z, F, O}.
 \end{enumerate}
\end{remark}
Before giving an example, let us make the following observation : In any commutative Hom-associative color algebra 
$(A, \cdot, \varepsilon, \alpha)$, one has,  for all $x, y, z\in \mathcal{H}(A)$,
\begin{eqnarray}
(x\cdot y)\cdot\alpha(z)=\alpha(x)\cdot(y\cdot z)=\varepsilon(y, z)\alpha(x)\cdot(z\cdot y)=\varepsilon(y,z)(x\cdot z)\cdot\alpha(y).\label{hcc}
\end{eqnarray}
\begin{example}
 Any commutative Hom-associative color algebra is a Hom-Novikov color algebra.
\end{example}
The below theorem allow to construct Hom-Novikov color algebras from a given one and a (weak) morphism.
\begin{theorem}\label{wm}
 Let $(A, \cdot, \varepsilon, \alpha)$ be a Hom-Novikov color algebra and let $\beta : A\rightarrow A$ be a weak morphism.
Then
 $$A_\beta=(A, \beta\circ\cdot, \beta\circ\alpha)$$
is also a Hom-Novikov color algebra. Moreover, if $A$ is multiplicative and $\beta$ is a morphism, then $A_\beta$ is also multiplicative.
\end{theorem}
\begin{proof}
 It suffises to apply $\beta^2$ to both side of equations (\ref{n1}) and (\ref{n2}). The second part is also easy to show.
\end{proof}
We have the following consequences.
\begin{corollary}
 Let $(A, \cdot, \varepsilon, \alpha)$ be a multiplicative Hom-Novikov color algebra, then, for any integer $n\geq0$,
$$A^n=(A, \alpha^n\circ\cdot, \varepsilon, \alpha^{n+1})$$
is also a multiplicative Hom-Novikov color algebra.
\end{corollary}
\begin{corollary}
 Let $(A, \cdot, \varepsilon)$ be a Novikov color algebra, and let $\beta : A\rightarrow A$ be a morphism.
Then
 $$A_\beta=(A, \beta\circ\cdot, \beta\circ\alpha)$$
is a multiplicative Hom-Novikov color algebra.
\end{corollary}

In the next result, we use centroids to provide Hom-Novikov color algebra from a given one.
\begin{proposition}
  Let $(A, \cdot, \varepsilon, \alpha)$ be a Hom-Novikov color algebra and let $\beta : A\rightarrow A$ be a centroid.
Then
 $$A'=(A, \beta\circ\cdot, \varepsilon, \alpha)$$
is a Hom-Novikov color algebra.
\end{proposition}
\begin{proof}
 For any $x, y, z\in\mathcal{H}(A)$,
 \begin{eqnarray}
  (x\cdot_\beta y)\cdot_\beta\alpha(z)
&=&\beta[\beta(x\cdot y)\cdot\alpha(z)]=\beta[(\beta(x)\cdot y)\cdot\alpha(z)]\nonumber\\
&=&\varepsilon(y, z)\beta[(\beta(x)\cdot z)\cdot\alpha(y)]\quad(\mbox{by}\;\;(\ref{n1})\nonumber\\
&=&\varepsilon(y, z)\beta[\beta(x\cdot z)\cdot\alpha(y)]\nonumber\\
&=&\varepsilon(y, z)(x\cdot_\beta z)\cdot_\beta\alpha(y).\nonumber
 \end{eqnarray}
Next,
\begin{eqnarray}
 &&\qquad(x\cdot_\beta y)\cdot_\beta\alpha(z)-\alpha(x)\cdot_\beta (y\cdot_\beta z)
-\varepsilon(x, y)\Big((y\cdot_\beta x)\cdot_\beta\alpha(z)-\alpha(y)\cdot_\beta (x\cdot_\beta z)\Big)=\nonumber\\
&&=\beta[\beta(x\cdot y)\cdot\alpha(z)]-\beta[\alpha(x)\cdot\beta (y\cdot z)]
-\varepsilon(x, y)\Big(\beta[\beta(y\cdot x)\cdot\alpha(z)]-\beta[\alpha(y)\cdot\beta (x\cdot z)]\Big)
\nonumber\\
&=&\beta[(\beta(x)\cdot y)\cdot\alpha(z)]-\beta[\alpha(x)\cdot(\beta (y)\cdot z)]
-\varepsilon(x, y)\Big(\beta[(\beta(y)\cdot x)\cdot\alpha(z)]-\beta[\alpha(y)\cdot(\beta (x)\cdot z)]
\nonumber\\
&=&(\beta(x)\cdot\beta(y))\cdot\alpha(z)-\beta(\alpha(x))\cdot(\beta (y)\cdot z)
-\varepsilon(x, y)\Big((\beta(y)\cdot\beta(x))\cdot\alpha(z)-\beta(\alpha(y))\cdot(\beta (x)\cdot z)
\nonumber\\
&=&(\beta(x)\cdot\beta(y))\cdot\alpha(z)-\alpha(\beta(x))\cdot(\beta (y)\cdot z)
-\varepsilon(x, y)\Big((\beta(y)\cdot\beta(x))\cdot\alpha(z)-\alpha(\beta(y))\cdot(\beta (x)\cdot z)\Big)
\nonumber
=0.\nonumber
\end{eqnarray}
The left hand side vanishes by (\ref{n2}).  Thus $A'$ is a Hom-Novikov color algebra.
\end{proof}

Now we have the following statement.
\begin{lemma}\label{nl}
 Let $(A, \cdot, \varepsilon, \alpha)$ be a Hom-Novikov color algebra. For all $x, y, z\in\mathcal{H}(A)$, we have
\begin{eqnarray}
 \varepsilon(z, x)[x, y]\cdot\alpha(z)+\varepsilon(x, y)[y, z]\cdot\alpha(x)+\varepsilon(y, z)[z, x]\cdot\alpha(y)&=&0,\label{ln1}\\
 \varepsilon(z, x)\alpha(x)\cdot[y, z]+\varepsilon(x, y)\alpha(y)\cdot[z, x]+\varepsilon(y, z)\alpha(z)\cdot[x, y]&=&0,
\end{eqnarray}
where $[x, y]=x\cdot y-\varepsilon(x, y)y\cdot x$.
\end{lemma}
\begin{proof}
 For any $x, y, z\in\mathcal{H}(A)$, one has
\begin{eqnarray}
 \varepsilon(z, x)[x, y]\cdot\alpha(z)=\varepsilon(z, x)(x\cdot y)\cdot\alpha(z)-\varepsilon(z, x)\varepsilon(x, y)(y\cdot x)\cdot\alpha(z),\nonumber
\end{eqnarray}
and
\begin{eqnarray}
 \sum_{x, y, z}\varepsilon(z, x)[x, y]\cdot\alpha(z)
&=&\varepsilon(z, x)(x\cdot y)\cdot\alpha(z)-\varepsilon(z, x)\varepsilon(x, y)(y\cdot x)\cdot\alpha(z)
+\varepsilon(x, y)(y\cdot z)\cdot\alpha(x)\nonumber\\
&&-\varepsilon(x, y)\varepsilon(y, z)(z\cdot y)\cdot\alpha(x)
+\varepsilon(y, z)(z\cdot x)\cdot\alpha(y)-\varepsilon(y, z)\varepsilon(z, x)(x\cdot z)\cdot\alpha(y)\nonumber\\
&=&\varepsilon(z, x)\Big((x\cdot y)\cdot\alpha(z)-\varepsilon(y, z)(x\cdot z)\cdot\alpha(y)\Big)
+\varepsilon(x, y)\Big((y\cdot z)\cdot\alpha(x)\nonumber\\
&&-\varepsilon(z, x)(y\cdot x)\cdot\alpha(z)\Big)+\varepsilon(y, z)\Big((z\cdot x)\cdot\alpha(y)\varepsilon(x, y)-(z\cdot y)\cdot\alpha(x)\Big).
\nonumber
\end{eqnarray}
The left hand side vanishes by (\ref{n1}).
Similarly, one has
\begin{eqnarray}
  \sum_{x, y, z} \varepsilon(z, x)\alpha(x)\cdot[y, z]
&=&\varepsilon(z, x)\alpha(x)\cdot(y\cdot z)-\varepsilon(z, x)\varepsilon(y, z)\alpha(x)\cdot(z\cdot y)
+\varepsilon(x, y)\alpha(y)\cdot(z\cdot x)\nonumber\\
&&-\varepsilon(x, y)\varepsilon(z, x)\alpha(y)\cdot(x\cdot z)
+\varepsilon(y, z)\alpha(z)\cdot(x\cdot y)-\varepsilon(y, z)\varepsilon(x, y)\alpha(z)\cdot(y\cdot x)\nonumber\\
&=&\varepsilon(z, x)\Big(\alpha(x)\cdot(y\cdot z)-\varepsilon(x, y)\alpha(y)\cdot(x\cdot z)\Big)
+\varepsilon(x, y)\Big(\alpha(y)\cdot(z\cdot x)\nonumber\\
&&-\varepsilon(y, z)\alpha(z)\cdot(y\cdot x)\Big)
+\varepsilon(y, z)\Big(\alpha(z)\cdot(x\cdot y)-\varepsilon(z, x)\alpha(x)\cdot(z\cdot y)\Big)\nonumber.
\end{eqnarray}
By axiom (\ref{n2}), 
\begin{eqnarray}
\sum_{x, y, z} \varepsilon(z, x)\alpha(x)\cdot[y, z]
&=&\varepsilon(z, x)\Big((x\cdot y)\cdot\alpha(z)-\varepsilon(x, y)(y\cdot x)\cdot\alpha(z)\Big)
+\varepsilon(x, y)\Big((y\cdot z)\cdot\alpha(x)\nonumber\\
&&-\varepsilon(y, z)(z\cdot y)\cdot\alpha(x)\Big)
+\varepsilon(y, z)\Big((z\cdot x)\cdot\alpha(y)-\varepsilon(z, x)(x\cdot z)\cdot\alpha(y)\Big)\nonumber\\
&=&\varepsilon(z, x)[x, y]\cdot\alpha(z)+\varepsilon(x, y)[y, z]\cdot\alpha(x)+\varepsilon(y, z)[z, x]\cdot\alpha(y).\nonumber
\end{eqnarray}
Which vanishes by (\ref{ln1}). Thus the conclusion holds.
\end{proof}
We are now ready to state the theorem.
\begin{theorem}\label{nlt}
  Let $(A, \cdot, \varepsilon, \alpha)$ be a Hom-Novikov color algebra. Then $HL(A)=(A, [-, -], \varepsilon, \alpha)$ is a Hom-Lie color algebra,
 with $[x, y]=x\cdot y-\varepsilon(x, y)y\cdot x$, for all $x, y \in\mathcal{H}(A)$.
\end{theorem}
\begin{proof}
 Let $x, y, z$ be any three homogeneous elements in $A$, 
\begin{eqnarray}
 \varepsilon(z, x)[\alpha(x), [y, z]]=\varepsilon(z, x)\alpha(x)\cdot[y, z]-\varepsilon(z, x)\varepsilon(x, y+z)[y, z]\cdot\alpha(x).\nonumber
\end{eqnarray}
Then
\begin{eqnarray}
\sum_{x, y, z} \varepsilon(z, x)[\alpha(x), [y, z]]
&=&\varepsilon(z, x)\alpha(x)\cdot[y, z]-\varepsilon(z, x)\varepsilon(x, y+z)[y, z]\cdot\alpha(x)
+\varepsilon(x, y)\alpha(y)\cdot[z, x]\nonumber\\
&&-\varepsilon(x, y)\varepsilon(y, z+x)[z, x]\cdot\alpha(y)
 +\varepsilon(y, z)\alpha(z)\cdot[x, y]-\varepsilon(y, z)\varepsilon(z, x+y)[x, y]\cdot\alpha(z)\nonumber\\
&=&    \varepsilon(z, x)\alpha(x)\cdot[y, z]+\varepsilon(x, y)\alpha(y)\cdot[z, x]+\varepsilon(y, z)\alpha(z)\cdot[x, y]\nonumber\\
&&-\varepsilon(x, y)[y, z]\cdot\alpha(x)-\varepsilon(y, z)[z, x]\cdot\alpha(y)-\varepsilon(z, x)[x, y]\cdot\alpha(z)\nonumber.
\end{eqnarray}
By Lemma \ref{nl}, the left hand side vanishes. 
\end{proof}

The below result is the color version of Proposition 3.4 and Proposition 3.5 of \cite{YY}.
\begin{proposition}\label{icn}
 If $(A, \cdot, \varepsilon, \alpha)$ is 
\begin{enumerate}
 \item [(1)]
an {\it involutive multiplicative} Hom-Novikov color algebra, then $(A, \alpha\circ\cdot, \varepsilon)$ is a Novikov color algebra.
\item [(2)]
a {\it regular} Hom-Novikov color algebra, then $(A, [-, -]_{\alpha^{-1}}={\alpha^{-1}}\circ [-, -], \varepsilon)$ is a Lie color algebra, where
$[x, y]=x\cdot y-\varepsilon(x, y)y\cdot x$, for all $x, y \in\mathcal{H}(A)$. In particular, if $\alpha$ is an {\it involution}, then 
$(A, [-, -]_{\alpha}={\alpha}\circ [-, -], \varepsilon)$ is a Lie color algebra.
\end{enumerate}
\end{proposition}
\begin{proof}
 It is straightforward.
\end{proof}
\begin{theorem}
 Let $(A, \cdot, \varepsilon, \alpha)$ be a commutative Hom-Novikov color algebra equiped with an averaging operator $\partial$. Then $A$ is a Hom-Novikov
color algebra with respect to the multiplication $\ast : A\times A\rightarrow A$ defined by
 $$x\ast y:=x\cdot\partial(y)$$
and the same twisting map $\alpha$.
\end{theorem}
\begin{proof}
We have to prove conditions (\ref{n1}) and (\ref{n2}) for $\ast$.  For any $x, y, z\in\mathcal{H}(A)$,
\begin{eqnarray}
 (x\ast y)\ast\alpha(z)
&=&(x\cdot\partial(y))\ast\alpha(z)
=(x\cdot \partial(y))\partial(\alpha(z))
=(x\cdot \partial(y))\alpha(\partial(z)
=\alpha(x)\cdot (\partial(y)\cdot\partial(z))\nonumber\\
&=&\varepsilon(y, z)\alpha(x)\cdot (\partial(z)\cdot\partial(y))
=\varepsilon(y, z)(x\cdot \partial(z))\cdot\alpha(\partial(y))
=\varepsilon(y, z)(x\cdot \partial(z))\cdot\partial(\alpha(y))\nonumber\\
&=&\varepsilon(y, z)(x\cdot \partial(z))\ast\alpha(y)
=\varepsilon(y, z)(x\ast z)\ast\alpha(y)\nonumber.
\end{eqnarray}
Next, 
\begin{eqnarray}
(x\ast y)\ast\alpha(z)-\alpha(x)\ast(y\ast z)
&=&\alpha(x)\cdot (\partial(y)\cdot\partial(z))-\alpha(x)\cdot\partial(y\cdot\partial(z))\nonumber\\
&=&\alpha(x)\cdot (\partial(y)\cdot\partial(z))-\alpha(x)\cdot(\partial(y)\cdot\partial(z))\nonumber\\
&=&0.\nonumber
\end{eqnarray}
Echanging the role of $x$ and $y$, it follows that
\begin{eqnarray}
 (x\ast y)\ast\alpha(z)-\alpha(x)\ast(y\ast z)=\varepsilon(x, y)\Big((y\ast x)\ast\alpha(z)-\alpha(y)\ast(x\ast z)\Big)=0.\nonumber
\end{eqnarray}
This completes the proof.
\end{proof}
\begin{theorem}\label{nk1}
 Let $(A, \cdot, \varepsilon, \alpha)$ be a commutative Hom-associative color algebra and $\partial: A\rightarrow A$ be an even derivation commutating
 with $\alpha$. Then the quadruple $(A, \ast, \varepsilon, \alpha)$ is a Hom-Novikov color algebra with respect to the product
$$x\ast y:=x\cdot \partial(y),$$
for all $x, y\in\mathcal{H}(A)$.
\end{theorem}
\begin{proof}
 For all $x, y, z\in \mathcal{H}(A)$, and by (\ref{hcc}),
\begin{eqnarray}
 (x\ast y)\ast\alpha(z)=(x\cdot\partial(y))\cdot(\partial(\alpha(z)))=(x\cdot\partial(y))\cdot\alpha(\partial(z))=
\varepsilon(y, z)(x\cdot\partial(z))\cdot\alpha(\partial(y))=\varepsilon(y, z) (x\ast z)\ast\alpha(y).\nonumber
\end{eqnarray}
For proving the second condition, one has 
\begin{eqnarray}
 (x\ast y)\ast\alpha(z)-\alpha(x)\ast(y\ast z)
&=&(x\cdot\partial(y))\cdot(\partial(\alpha(z)))-\alpha(x)\cdot\partial(y\cdot\partial(z))\nonumber\\
&=&(x\cdot\partial(y))\cdot(\partial(\alpha(z)))-\alpha(x)\cdot(\partial(y)\cdot\partial(z))-\alpha(x)\cdot (y\cdot\partial^2(z))\nonumber\\
&=&-\alpha(x)\cdot (y\cdot\partial^2(z))\nonumber.
\end{eqnarray}
And, by echanging the role of $x$ and $y$,
\begin{eqnarray}
 (y\ast x)\ast\alpha(z)-\alpha(y)\ast(x\ast z)=-\alpha(y)\cdot (x\cdot\partial^2(z))\nonumber.
\end{eqnarray}
Then
\begin{eqnarray}
 (x\ast y)\ast\alpha(z)-\alpha(x)\ast(y\ast z)
&=&-(x\cdot y)\cdot(\alpha(\partial^2(z)))\nonumber\\
&=&-\varepsilon(x, y)(y\cdot x)\cdot(\alpha(\partial^2(z)))\nonumber\\
&=&-\varepsilon(x, y)\alpha(y)\cdot (x\cdot\partial^2(z))\nonumber\\
&=&\varepsilon(x, y)\Big((y\ast x)\ast\alpha(z)-\alpha(y)\ast(x\ast z)\Big)\nonumber.
\end{eqnarray}
This ends the proof.
\end{proof}
\begin{corollary}
 Let $(A, \cdot, \varepsilon)$ be a commutative associative color algebra, $\alpha : A \rightarrow A$ be an algebra morphism and
$D$ be a derivation of $A$ commutating with $\alpha$. Then  $(A, \ast, \varepsilon, \alpha)$ is a Hom-Novikov color algebra, where
$$x\ast y:=\alpha(x\cdot\partial(y))$$
for $x, y\in\mathcal{H}(A)$.
\end{corollary}
\begin{proof}
 It comes from theorem \ref{nk1} and Theorem \ref{wm} for $\alpha=id_A$.
\end{proof}
For a Hom-Lie color algebra $(L, [-, -], \varepsilon, \alpha)$,  write $Z(\alpha(L))$ for the subset $L$ consisting of elements  $x$ such that $[x, \alpha(y)]=0$ for all $x, y\in L$.
Clearly, if $L$ is a Lie color algebra (i.e. $\alpha=id$), then  $Z(id(L))=Z(L)$ is the center of $L$.
\begin{proposition}
 Let $(L, [-, -], \varepsilon, \alpha)$ be a (multiplicative) Hom-Lie color algebra and $f : L\rightarrow L$ an even linear map such that
$f\circ\alpha=\alpha\circ f$. Define the product $\ast : A\times A\rightarrow A$ by
$$x\ast y=[f(x), y].$$
Then $(L, \ast, \varepsilon, \alpha)$ is a Hom-Novikov color algebra if and only if the following conditions hold for any
$x, y, z\in\mathcal{H}(A)$ :
\begin{eqnarray}
 f([f(x), y]+[x, f(y)])-[f(x), f(y)]\in Z(\alpha(L)) \label{f}
\end{eqnarray}
and 
\begin{eqnarray}
 [f([f(x), y]), \alpha(z)]=\varepsilon(y, z)[f([f(x), z]), \alpha(y)] \label{g}
\end{eqnarray}
\end{proposition}
\begin{proof}
 It is clear that $\alpha$ is a morphism with respect to $\ast$. Then
the condition (\ref{n2}) for the multiplication $\ast$ means
$$ [f([f(x), y]), \alpha(z)]-[f(\alpha(x)), [f(y), z]]=\varepsilon(x, y)[f([f(y), x]), \alpha(z)]-\varepsilon(x, y)[f(\alpha(y)), [f(x), z]].$$
Using the fact that $f$commutes with $\alpha$ and the $\varepsilon$-skew-symmetry of the bracket $[-, -]$,
\begin{eqnarray}
[f([f(x), y])+f([x, f(y)]), \alpha(z)]
&=&[\alpha(f(x)), [f(y), z]]+\varepsilon(x, y)\varepsilon(x, z)[\alpha(f(y)), [z, f(x)]]\nonumber\\
&=&-\varepsilon(x +y, z)[\alpha(z),[f(x), f(y)]].\nonumber
\end{eqnarray}
The last equality follows form the $\varepsilon$-Hom-Jacobi identity. Thus, condition (\ref{n2}) holds for the multiplication $\ast$ if and only if
$$[f([f(x), y]+[x, f(y)])-[f(x), f(y)], \alpha(z)]=0.$$
Which is equivalent to (\ref{f}). The condition (\ref{g}) is simply a restatement of (\ref{n1}) for the multiplication $\ast$.
This ends the proof.
\end{proof}
\begin{corollary}
 Let $(L, [-, -], \varepsilon, R, \alpha)$ be a (multiplicative) Hom-Lie Rota-Baxter color algebra of weight zero such that
$$[R([R(x), y]), \alpha(z)]=\varepsilon(y, z)[R([R(x), z]), \alpha(y)].$$
Then $(L, \ast, \varepsilon, \alpha)$ is a Hom-Novikov color algebra.
\end{corollary}
\begin{proof}
 It comes from Lemma \ref{hde}.
\end{proof}
The following proposition is proved by direct computation.
\begin{proposition}
 Given two Hom-Novikov algebra $(A_1, \cdot_1, \alpha_1)$ and $(A_2, \cdot_2, \alpha_2)$, there is a Hom-Novikov algebra 
$(A_1\oplus A_2, \cdot, \alpha_1\oplus \alpha_2)$, where 
\begin{eqnarray}
 (\alpha_1\oplus\alpha_2)(x_1+x_2)&:=&\alpha_1(x_1)+\alpha_2(x_2)\\
(x_1, x_2)\cdot(y_1, y_2)&:=&(x_1\cdot_1 y_1, x_2\cdot_2 y_2)\\
\end{eqnarray}
\end{proposition}

Now we have the following theorem.
\begin{theorem}
 Let $(A, \cdot, \varepsilon, \alpha_A)$ be a commutative 
Hom-associative color algebra and $(S, \ast, \varepsilon, \alpha_S)$ be a Hom-Novikov color algebra.
Then  $(S\otimes A, \star, \alpha_{S\otimes A})$ is a Hom-Novikov color algebra, with
\begin{eqnarray}
\alpha_{S\otimes A}&=&\alpha_S\otimes\alpha_A,\nonumber\\
 (x\otimes a)\star(y\otimes b)&=&\varepsilon(a, y)(x\ast y)\otimes(a\cdot b),\nonumber
\end{eqnarray}
for all $x, y\in \mathcal{H}(S), a, b\in \mathcal{H}(A)$.
\end{theorem}
\begin{proof}
To simplify the typography, we omit the subscripts $A$ and $S$ and write $\alpha(x)$ and $\alpha(a)$ instead of $\alpha_S(x)$ and $\alpha_A(a)$
respectively. First, for all $x, y, z\in \mathcal{H}(S), a, b, c\in \mathcal{H}(A)$, we have
\begin{eqnarray}
[(x\otimes a)\star(y\otimes b)]\star(\alpha(z)\otimes \alpha(c))
&=&\varepsilon(a, y)[(x\ast y)\otimes(a\cdot b)]\star(\alpha(z)\otimes\alpha(c))\nonumber\\
&=&\varepsilon(a+b, z)\varepsilon(a, y)(x\ast y)\ast\alpha(z)\otimes(a\cdot b)\cdot\alpha(c)\nonumber\\
&=&\varepsilon(a, z)\varepsilon(b, z)\varepsilon(a, y)(x\ast y)\ast\alpha(z)\otimes(a\cdot b)\cdot\alpha(c),\nonumber
\end{eqnarray}
and 
\begin{eqnarray}
 (\alpha(x)\otimes\alpha(a))\star[(y\otimes b)\star(z\otimes c)]
&=&\varepsilon(b, z)(\alpha(x)\otimes\alpha(a))\star[(y\ast z)\otimes(b\cdot c)]\nonumber\\
&=&\varepsilon(b, z)\varepsilon(a, y+z)\alpha(x)\ast(y\ast z)\otimes\alpha(a)\cdot(b\cdot c)\nonumber\\
&=&\varepsilon(b, z)\varepsilon(a, y)\varepsilon(a, z)\alpha(x)\ast(y\ast z)\otimes\alpha(a)\cdot(b\cdot c)\nonumber.
\end{eqnarray}
 Next, let us prove axiom (\ref{n1}) for $\star$. 
\begin{eqnarray}
&& \quad[(x\otimes a)\star(y\otimes b)]\star(\alpha(z)\otimes \alpha(c))
-\varepsilon(y+b, z+c)[(x\otimes a)\star(z\otimes c)]\star(\alpha(y)\otimes \alpha(b))=\nonumber\\
&&=\varepsilon(a, z)\varepsilon(b, z)\varepsilon(a, y)(x\ast y)\ast\alpha(z)\otimes(a\cdot b)\cdot\alpha(c)\nonumber\\
&&\quad\quad-\varepsilon(y+b, z+c)\varepsilon(a, y)\varepsilon(c, y)\varepsilon(a, z)(x\ast z)\ast\alpha(y)\otimes(a\cdot c)\cdot\alpha(b)\nonumber\\
&&=\varepsilon(a, z)\varepsilon(b, z)\varepsilon(a, y)\varepsilon(y, z)\varepsilon(b, c)(x\ast z)\ast\alpha(y)\otimes(a\cdot c)\cdot\alpha(b)
\;\;(\mbox{by}\;\; (\ref{n1})\;\;\mbox{and}\;\;(\ref{hcc}))\nonumber\\
&&\quad\quad-\varepsilon(y, z)\varepsilon(y, c)\varepsilon(b, z)\varepsilon(b, c)
 \varepsilon(a, y)\varepsilon(c, y)\varepsilon(a, z)(x\ast z)\ast\alpha(y)\otimes(a\cdot c)\cdot\alpha(b)\nonumber\\
&&=0.\nonumber
\end{eqnarray}
Now let us prove axiom (\ref{n2}) for $\star$. Let us set 
\begin{eqnarray}
D(x, a, y, b, z, c)
&=&[(x\otimes a)\star(y\otimes b)]\star(\alpha(z)\otimes \alpha(c))
-(\alpha(x)\otimes\alpha(a))\star[(y\otimes b)\star(z\otimes c)]\nonumber\\
&=&\varepsilon(a, z)\varepsilon(b, z)\varepsilon(a, y)(x\ast y)\ast\alpha(z)\otimes(a\cdot b)\cdot\alpha(c)\nonumber\\
&&-\varepsilon(b, z)\varepsilon(a, y)\varepsilon(a, z)\alpha(x)\ast(y\ast z)\otimes\alpha(a)\cdot(b\cdot c)\nonumber\\
&=&\varepsilon(a, z)\varepsilon(b, z)\varepsilon(a, y)[(x\ast y)\ast\alpha(z)- \alpha(x)\ast(y\ast z)]\otimes(a\cdot b)\cdot\alpha(c)
 \;\;(\mbox{by}\;\; (\ref{hass}))\nonumber.
\end{eqnarray}
Then
\begin{eqnarray}
 &&\quad D(x, a, y, b, z, c)-\varepsilon(x+a, y+b)D(y, b, x, a, z, c)=\nonumber\\
&&=\varepsilon(a, z)\varepsilon(b, z)\varepsilon(a, y)[(x\ast y)\ast\alpha(z)- \alpha(x)\ast(y\ast z)]\otimes(a\cdot b)\cdot\alpha(c)\nonumber\\
&&\quad-\varepsilon(x+a, y+b)\varepsilon(b, z)\varepsilon(a, z)\varepsilon(b, x)[(y\ast x)\ast\alpha(z)
- \alpha(y)\ast(x\ast z)]\otimes(b\cdot a)\cdot\alpha(c)\nonumber.
\end{eqnarray}
By $\varepsilon$-commutativity, 
\begin{eqnarray}
 &&\quad D(x, a, y, b, z, c)-\varepsilon(x+a, y+b)D(y, b, x, a, z, c)=\nonumber\\
&&=\varepsilon(a, b)\varepsilon(a, z)\varepsilon(b, z)\varepsilon(a, y)[(x\ast y)\ast\alpha(z)
- \alpha(x)\ast(y\ast z)]\otimes(b\cdot a)\cdot\alpha(c)\nonumber\\
&&\quad-\varepsilon(x+a, y+b)\varepsilon(b, z)\varepsilon(a, z)\varepsilon(b, x)[(y\ast x)\ast\alpha(z)
- \alpha(y)\ast(x\ast z)]\otimes(b\cdot a)\cdot\alpha(c)\nonumber.
\end{eqnarray}
The right hand side vanishes by (\ref{n2}). 
\end{proof}

\section{Hom-quadratic Hom-Novikov color algebras}
Likewise Hom-quadratic Hom-associative color algebras,  Hom-quadratic Hom-Lie color algebras \cite{FISA} and quadratic Hom-Novikov algebra \cite{YY}, 
we introduce here the Hom-quadratic Hom-Novikov color algebras and study their various twisting.
\begin{definition}
 A Hom-quadratic Hom-Novikov color algebra $(A, \cdot, \varepsilon, B, \alpha, \beta)$ is a Hom-Novikov color algebra
$(A, \cdot, \varepsilon,\alpha)$ with a pair $(B, \beta)$ where
\begin{enumerate}
 \item [a)]
$\beta : A \rightarrow A$ is an even linear map,
\item  [b)]
$\varepsilon$-symmetric i.e. $B(x, y)=\varepsilon(x, y)B(y, x)$, 
\item [c)]
  $B$ is a nondegenerate,
\item  [d)]
invariant bilinear form on $A$ i.e. $B(x\cdot y, \beta(z))=B(\beta(x), y\cdot z)$, and 
\item  [e)]
$\alpha$ is $B$-symmetric i.e. $B(\alpha(x), y)=B(x, \alpha(y))$
\end{enumerate}
When $\beta=id_A$, $B$ is said to be invariant and $(A, \cdot, \varepsilon, \alpha, B, id)$ is called quadratic Hom-Novikov color algebra and simply
denoted $(A, \cdot, \varepsilon, \alpha, B)$.
\end{definition}
Let $Aut_s(A, B)$ be the set of all  automorphisms $\varphi$ of $A$ which satisfy $B(\varphi(x), y)=B(x, \varphi(y)$, for all $x, y\in A$.
 We call $Aut_s(A, B)$ the  set of symmetric automorphisms of $A$.

The below result allows to obtain a quadratic Hom-Novikov color algebra from a given one and a symmetric automorphism. That is
the subcategory of quadratic Hom-Novikov color algebras is closed under symmetric automorphisms.
\begin{theorem}
 Let $(A, \cdot, \varepsilon, \alpha, B)$ be a quadratic Hom-Novikov color algebra and $\beta$ be a symmetric automorphism.
Then 
$$A_\beta=(A, \cdot_\beta=\beta\circ\cdot, \varepsilon, \beta\circ\alpha, B_\beta(-, -)=B(\beta(-), -))$$
 is a quadratic Hom-Novikov color algebra.
\end{theorem}
\begin{proof}
 It is clear that $B_\beta$ is a bilinear form, $\varepsilon$-symmetric, nondegenerate (since $\beta$ is an automorphism). Now prove that 
$B_\beta$ is invariant under the multiplication $\cdot_\beta$. For any $x, y, z \in\mathcal{H}(A)$,
\begin{eqnarray}
 B_\beta(x\cdot_\beta y, z)&=&B_\beta(\beta(x)\cdot\beta(y), z)=B(\beta(\beta(x)\cdot\beta(y)), z)\nonumber\\
&=&B(\beta(x)\cdot\beta(y), \beta(z))
=B(\beta(x), \beta(y)\cdot \beta(z))=B(\beta(x), y\cdot_\beta z)\nonumber\\
&=&B_\beta(x, y\cdot_\beta z).\nonumber
\end{eqnarray}
Finally, $\beta\circ\alpha$ is B-symmetric. In fact,
\begin{eqnarray}
 B_\beta((\beta\circ\alpha)(x), y)&=&B((\beta^2\circ\alpha)(x), y)=B((\beta\circ\alpha\circ\beta)(x), y)\nonumber\\
&=&B(\beta(x), (\alpha\circ\beta)(y))=B(\beta(x), (\beta\circ\alpha)(y))=B_\beta(x, (\beta\circ\alpha)(y))\nonumber.
\end{eqnarray}
This completes the proof.
\end{proof}
\begin{corollary}
 Let $(A, \cdot, \varepsilon, \alpha, B)$ be a quadratic regular Hom-Novikov color algebra. Then
$(A, \alpha^n\circ\cdot, \varepsilon, \alpha^{n+1}, B_{\alpha^n})$ is a regular Hom-Novikov color algebra and
$(A, \alpha^n\circ\cdot, \varepsilon, \alpha^{n+1}, B_{\alpha})$, where $B_\alpha(x, y)=B(\alpha^n(x), y)$, is a quadratic regular
Hom-Novikov color algebra.
\end{corollary}
\begin{proof}
 It follows from Theorem \ref{wm} and taking $\beta=\alpha^n$.
\end{proof}

From Hom-quadratic Hom-Novikov color algebra $A$ to the Hom-quadratic Hom-Lie color algebra associated to $A$. 
\begin{proposition}
 Let $(A, \cdot, \varepsilon, \alpha, B, \alpha)$ be a Hom-quadratic regular Hom-Novikov color algebra. 
Then $(A, [-, -], \varepsilon, \alpha, B_\alpha(-, -)=B(\alpha(-), -), \alpha)$
 is a Hom-quadratic  Hom-Lie color algebra, where $[x, y]=x\cdot y-\varepsilon(x, y)y\cdot x$ for all $x, y\in\mathcal{H}(A)$.
\end{proposition}
\begin{proof}
 Since $B$ is nondegenerate and $\alpha$ is an automorphism, $B_\alpha$ is nondegenerate. The $\varepsilon$-symmetric of $B_\alpha$ and 
the fact that $\alpha$ is $B$-symmetric are trivial. Thus
\begin{eqnarray}
 B_\alpha([x, y], \alpha(z))
&=&B([\alpha(x), \alpha(y)], \alpha(z))\nonumber\\
&=&B(\alpha(x)\cdot \alpha(y), \alpha(z))-\varepsilon(x, y)B(\alpha(y)\cdot \alpha(x), \alpha(z))\nonumber\\
&=&B(\alpha(x), \alpha(y)\cdot \alpha(z))-\varepsilon(y, z) B(\alpha(x), \alpha(z)\cdot \alpha(y))\nonumber\\
&=&B(\alpha^2(x), y\cdot z)-\varepsilon(y, z) B(\alpha^2(x), z\cdot y)\nonumber\\
&=&B_\alpha(\alpha(x), y\cdot z)-\varepsilon(y, z) B_\alpha(\alpha(x), z\cdot y)\nonumber\\
&=&B_\alpha(\alpha(x), [y, z])\nonumber
\end{eqnarray}
for all $x, y, z\in\mathcal{H}(A)$.
\end{proof}
\begin{corollary}
 Let $(A, \cdot, \varepsilon,  B)$ be a quadratic Novikov color algebra and $\alpha : A\rightarrow A$ be an automorphism of $A$ which is $B$-symmetric.
Then $(A, [-, -]_\alpha=\alpha\circ[-, -], \alpha, B_\alpha)$ is a quadratic Hom-Lie color algebra.
\end{corollary}

From quadratic Hom-Novikov color algebra to the associated quadratic Hom-Lie color algebra. 
\begin{theorem}
 Let $(A, \cdot, \varepsilon, \alpha, B)$ be a quadratic Hom-Novikov color algebra. Then $(A, [-, -], \varepsilon, \alpha, B)$
 is a quadratic  Hom-Lie color algebra, where $[x, y]=x\cdot y-\varepsilon(x, y)y\cdot x$ for all $x, y\in\mathcal{H}(A)$.
\end{theorem}
\begin{proof}
Thanks to Theorem \ref{nlt}, $(A, [-, -], \varepsilon, \alpha)$ is a Hom-Lie color algebra. We have to prove that $B$ is invariant under the bracket
$[-, -]$. For all $x, y, z \in\mathcal{H}(A)$,
\begin{eqnarray}
 B([x, y], z)
&=&B(x\cdot y-\varepsilon(x, y)y\cdot x, z)\nonumber\\
&=&B(x\cdot y, z)-\varepsilon(x, y)B(y\cdot x, z)\nonumber\\
&=&B(x, y\cdot z)-\varepsilon(x, y)\varepsilon(x+y, z) B(z, y\cdot x)\nonumber\\
&=&B(x, y\cdot z)-\varepsilon(x, y)\varepsilon(x+y, z) B(z\cdot y, x)\nonumber\\
&=&B(x, y\cdot z)-\varepsilon(x, y)\varepsilon(x+y, z)\varepsilon(y+z, x) B(x, z\cdot y)\nonumber\\
&=&B(x, y\cdot z)-\varepsilon(y, z)B(x, z\cdot y)\nonumber\\
&=&B(x, y\cdot z-\varepsilon(y, z)z\cdot y)\nonumber\\
&=&B(x, [y, z])\nonumber.
\end{eqnarray}
This ends the proof.
\end{proof}
From Hom-quadratic Hom-Novikov color algebras to quadratic Novikov color algebras.
\begin{proposition}
 Let $(A, \cdot, \varepsilon, \alpha, B, \alpha)$ be a Hom-quadratic involutive Hom-Novikov color algebra. 
Then $(A, \cdot_\alpha=\alpha\circ\cdot, \varepsilon, B)$ is a quadratic Novikov color algebra.
\end{proposition}
\begin{proof}
 By Proposition \ref{icn}, $(A, \cdot_\alpha=\alpha\circ\cdot, \varepsilon)$ is a Novikov color algebra. It suffises to prove that $B$ is
 invariant under $\cdot_\alpha$. For any $x, y, z\in\mathcal{H}(A)$.
\begin{eqnarray}
 B(x, y\cdot_\alpha z)=B(x, \alpha(y\cdot z))=B(\alpha(x), y\cdot z)=B(x\cdot y, \alpha(z))=B(\alpha(x\cdot y), z)=B(x\cdot_\alpha y, z),\nonumber
\end{eqnarray}
which ends the proof.
\end{proof}

\label{lastpage-01}
\end{document}